\documentclass[a4paper,11pt]{amsart}

\usepackage[english]{my-shortcuts}
\usepackage{srcltx,algorithm,algorithmic}
\usepackage{a4wide}

\begin{document}

\title[Perturbation of extreme singular values after appending a column]
{On the perturbation of the extremal singular values of a matrix after appending a column}
\author{St\'ephane Chr\'etien} 
\author{S\'ebastien Darses}

\address{Laboratoire de Math\'ematiques, UMR 6623\\ 
Universit\'e de Franche-Comt\'e, 16 route de Gray,\\
25030 Besancon, France} 
\email{stephane.chretien@univ-fcomte.fr}

\address{LATP, UMR 6632\\
Universit\'e Aix-Marseille, Technop\^ole Ch\^{a}teau-Gombert\\
39 rue Joliot Curie\\ 13453 Marseille Cedex 13, France \\
and \\
Laboratoire de Math\'ematiques, UMR 6623\\ 
Universit\'e de Franche-Comt\'e, 16 route de Gray,\\
25030 Besancon, France}
\email{sebastien.darses@univ-amu.fr}

\maketitle

\vspace{.5cm}

\begin{abstract}
We first review various bounds on the extreme singular values of a matrix in the particular case where it is obtained after appending a column vector to a given 
matrix. Most of the results are contained in \cite{Li-Li}. We provide simple proofs based on the study of the characteristic polynomials rather than variational methods. Second, we present three applications 
to random matrix theory, signal processing and control theory.
\end{abstract}


\section{Introduction}

\subsection{Framework}

Let $d$ be an integer. Let $X\in\R^{d\times n}$ be a $d\times n$-matrix and let $x\in \R^{d}$ be column vector.  We denote by a subscript $^t$ the transpose of vectors and matrices. 
There exist at least two ways to study the singular values of the matrix $(x,X)$ obtained by appending the column vector $x$ to the matrix $X$:
\begin{enumerate}
\item[{\bf (A1)}] Consider  the matrix 
\bea \label{add}
A & = & 
\left[
\begin{array}{c}
x^t \\
X^t
\end{array}
\right]
\left[
\begin{array}{cc}
x & X
\end{array}
\right]
=
\left[
\begin{array}{cc}
x^tx & x^tX \\
X^tx & X^tX
\end{array}
\right];
\eea
\item[{\bf (A2)}] Consider  the matrix 
\bean
\wdt{A} & = & 
\left[
\begin{array}{cc}
x & X
\end{array}
\right]\left[ 
\begin{array}{c}
x^t \\
X^t
\end{array}
\right]
= XX^t+xx^t.
\eean 
\end{enumerate}

On one hand, one may study in {\bf (A1)} the eigenvalues of the $(n+1)\times (n+1)$ hermitian matrix $A$, i.e. the matrix $X^tX$ augmented with an arrow matrix.

On the other hand, one will deal in {\bf (A2)}  with the eigenvalues of the $d\times d$ hermitian matrix $\wdt A$, which may be seen as a rank-one perturbation of $XX^t$. The matrices $A$ and $\wdt A$ have the same non-zeros eigenvalues, and in particular 
$\lb_{\max}(A) = \lb_{\max}(\wdt A)$.
Moreover, the singular values of 
the matrix $(x,X)$ are the square-root of the eigenvalues of the matrix $A$.

Equivalently, the problem of a rank-one perturbation can be rephrased as 
the one of controlling the perturbation of the singular values of a matrix after appending a column. 
\\

In this paper, we study a slightly more general framework than  {\bf (A1)}, that is the case of a matrix
\bea
\label{A}
A & = & 
\left[
\begin{array}{cc}
c & a^t \\
a & M
\end{array}
\right],
\eea
where $a\in \R^{d}$, $c\in \R$ and $M\in \R^{d\times d}$ is a symmetric matrix.

Our goal is to present new bounds on the extreme eigenvalues of $A$ as a function of 
the eigenvalues of $M$ and the norm of $a$, and we will focus on various applications. 
Indeed, this problem occurs in a variety of contexts such as the 
perturbation analysis of covariance matrices in statistics \cite{Nadler:AnnStat08}, 
the study of the Restricted Isometry Constant in Compressed Sensing \cite{Candes:CRAS10}, 
spectral graph theory and edge deletion \cite{BrouwerHaemers:Springer12}, 
control theory of complex networks \cite{PorfiriDiBernardo:Automatica08}, 
hitting time analysis for classical or quantum random walks \cite{ZhangEtAl:IEEETransComm12}, 
robust face recognition \cite{Qiu:IntJourInnovComp11}, wireless comunications \cite{ShenSuter:EURASIP09}, 
communication theory and signal processing \cite{ZhangEtAl:IEEETransComm12}, 
numerical methods for partial differential equations \cite{BlankEtAl:JCompPhys12}, 
numerical analysis of bifurcations \cite{DicksonEtAl:SIAMNumAl07},
among many applications. 

Notice further that in (\ref{A}) if $M$ and $A$ are positive definite, there exist $X\in\R^{d\times n}$ and $x\in \R^{d}$ such that $M=X^tX$ and $A$ can be written as in (\ref{add}) due to the Cholesky decomposition.

\subsection{Additional notations} The Kronecker symbol is denoted by $\delta_{i,j}$, i.e. $\delta_{i,j}=1$ if $i=j$ and
is equal to zero otherwise. We denote by $\|x\|_2$ the euclidian norm of a vector $x$ and by $\|A\|$ the associated operator norm (spectral norm) of a matrix $A$.

For any symmetric matrix $B\in \R^{d\times d}$ we will denote its eigenvalues by $\lb_1(B)
\ge \cdots \ge \lb_d(B)$. The largest eigenvalue will sometimes also be denoted by $\lambda_{\max}(B)$ and 
the smallest by $\lb_{\min}(B)$. The smallest nonzero eigenvalue of a positive semi-definite matrix $B$ will be denoted by $\lb_{\min>0}(B)$.

\subsection{Plan of the paper}
Section \ref{previous} is devoted to an overview of known results. Section \ref{main-results} presents new upper and lower bounds for the extreme eigenvalues. Section \label{normop} translates some previous results in terms of operator norm together with a slight variation. Finally, Section \ref{app} is concerned with the applications in Compressed sensing and graphs theory.

\section{Previous results on eigenvalue perturbation} \label{previous}

We now review some previous, old and recent results from matrix perturbation theory and 
apply them to our problem of appending a column. 

Obtaining precise estimates on the eigenvalues of a sum of two matrices (say $X+P$, considering $P$ as a perturbation) is a very difficult task in general. Weyl's and Horn's inequalities for instance can be employed and these bounds can be improved when knowing that the perturbation $P$ is small with respect to $X$ (see e.g. \cite[Chap. 6]{horn}). The whole point of the works \cite{batson} and \cite{benaych}, to name a few, is to understand how randomness can simplify this analysis.

\subsection{Weyl's inequalities}
The reference \cite{Tao:Blog10} gives an 
overview of many inequalities on the eigenvalues of sums of symmetric (and Hermitian) matrices.  
The Weyl inequalities are given as follows:
\begin{theo}[Weyl] Let $B$ and $B^\prime$ be symmetric real matrices in $\R^{d\times d}$ and let $\lb_j(B)$, $j=1,\ldots,d$, (resp. $\lb_j(B^\prime)$), denote the eigenvalues of $B$ (resp. $B^\prime$). Then, we have
\bean
\lb_{i+j-1} (B+B^\prime) & \le & \lb_i(B)+\lb_j(B^\prime),
\eean 
whenever $i,j\ge 1$ and $i+j-1\le n$. 
\label{weyl}
\end{theo}
\subsubsection{The arrowhead perturbation}
Consider the case where we would like to control the largest eigenvalue of $A$ with the eigenvalues of $M=X^tX$. We 
have the following result. 
\begin{prop}
We have 
\bean 
\lb_1(A) & \le & \max \{c,\lb_1(M)\} + \| a\|_2. 
\eean 
\end{prop}
\begin{proof}
The Weyl inequalities for $i,j=1$ gives that 
\bea
\label{weyl1}
\lb_1(A) & \le & \lb_1\left(\left[
\begin{array}{cc}
c & 0 \\
0 & M
\end{array}
\right]\right)
+\lb_1(E)
\eea
with 
\bean 
E =
\left[
\begin{array}{cc}
0 & a^t \\
a & 0
\end{array}
\right].
\eean 
Moreover, using the variational representation of the maximum eigenvalue and the method of Lagrange multipliers, 
we have
$\lb_1(E) = \| a\|_2$.
Combining this with (\ref{weyl1}), we obtain the desired result. 
\end{proof}
The main fact to retain from this inequality is that if $x$ is orthogonal to all columns of 
$X$, then $a=0$ and the perturbation has no effect on the largest eigenvalue as long as $c\le \lb_1(M)$.
This elementary observation can be extrapolated to much more difficult situations, e.g. in
the spiked covariance model where a phase transition has been proved between concerning the ability to detect 
a spike or not, depending on the energy level of the spike \cite[Theorem 2.3]{Nadler:AnnStat08}. 

\subsubsection{The rank-one perturbation}
If we only want to study the perturbation of the largest eigenvalue, then we can consider 
the rank-one perturbation described by {\bf (A2)}. In this case, Weyl's bound gives the following 
result. 
\begin{prop}
We have 
\bean 
\lb_1(A) & \le & \lb_1(M) + \|x\|_2^2. 
\eean 
\end{prop} 
\begin{proof}
Set 
$\wdt{A}=\wdt{M}+xx^t$. Using that $\lb_1(A)=\lb_1(\wdt{A})$ and $\lb_1(M)=\lb_1(\wdt{M})$, we obtain from 
Theorem \ref{weyl} : 
\bean 
\lb_1(A) & \le & \lb_1(M)+\lb_1(xx^t). 
\eean 
Since $\lb_1(xx^t)=\|x\|_2^2$, the conclusion follows. 
\end{proof}
The main drawback of this inequality is that it does not take into account the geometry of the problem and 
in particular the angle between $X$ and the new vector $x$ that we want to append to $X$. This does not 
disqualify the rank-one perturbation approach to controlling the maximum eigenvalue as  
will be shown in Subsection \ref{ipsen}.

\subsection{An inequality of Li and Li}

They prove a general inequality concerning the perturbation of eigenvalues under off-block diagonal perturbations. We specify their result, \cite[Theorem 2]{Li-Li}, in our context:
\bea \label{li-li}
\left|\lb_1(A)-\max(c,\lb_1(M))\right| & \le & \frac{2\|a\|_2^2}{\eta_1+\sqrt{\eta_1^2+4\|a\|_2^2}},
\eea
with $\eta_1=|c-\lb_1(M)|$. In their paper, $\wdt \lb_1$ is actually $\max(c,\lb_1(M))$ here.

We refer to \cite{Li-Li} and references therein for the history of such inequalities.

\subsection{An inequality of Ipsen and Nadler} \label{ipsen}
In \cite{IpsenNadler:SIAMMA09}, the authors propose a bound for the eigenvalues of 
$\wdt{A}$ in the problem of rank one perturbation {\bf (A2)}. The following 
theorem is a corollary of their main result where we restrict our attention to the largest 
eigenvalue. 
\begin{theo}
Let $\wdt{M}\in \C^{d\times d}$ denote an Hermitian matrix and let $x\in \C^d$. Let $V_1$ (resp. $V_2$) denote the 
eigenvector associated to the eigenvalue $\lb_1(\wdt{M})$ (resp. $\lb_2(\wdt{M})$). Let 
$\wdt{A}=\wdt{M}+xx^t$. Then 
\bean 
\lb_1(\wdt{M}) +\delta_{\min} & \le \lb_1(\wdt{A}) \le &  \lb_1(\wdt{M}) +\delta_{\max}, 
\eean  
with 
\bean 
\delta_{\min} & = & \frac12\left(\|P_{\langle (V_1,V_2)\rangle}(x)\|_2^2-{\rm gap}_2+\sqrt{({\rm gap}_2+\|P_{\langle (V_1,V_2)\rangle}(x)\|_2^2)^2-4 \ {\rm gap}_2 \|P_{\langle (V_2)\rangle}(x)\|_2^2} \right) \\
\delta_{\max} & = & \frac12\left(\|x\|_2^2-{\rm gap}_2+\sqrt{({\rm gap}_2+\|x\|_2^2)^2-4 \ {\rm gap}_2 \|P_{\langle (V_2,\ldots,V_d)\rangle}(x)\|_2^2} \right),
\eean 
where $(V_i,\ldots,V_j)$, $1\le i\le j\le d$, denotes the vector space generated by $V_i,\ldots,V_j$ and 
$P_{\langle (V_i,\ldots,V_j)\rangle}$ denotes the orthogonal projection onto this space, and  
\bean 
{\rm gap}_2 & = & \lb_1(\wdt M)-\lb_2(\wdt M). 
\eean 
\end{theo}
This inequality has been used in various applications such as control of complex systems 
\cite{PorfiriDiBernardo:Automatica08}, quantum information theory \cite{ChiangEtAl:QuantInfProc13},
communication theory and signal processing \cite{ZhangEtAl:IEEETransComm12}, numerical methods for partial differential 
equations \cite{BlankEtAl:JCompPhys12}. One drawback of using this result in our context is that we have 
to know the spacing ${\rm gap}_2$ for the second eigenvalue. Moreover, the upper bound 
depends on $\|x\|_2^2$ and does not take into account the scalar products of $x$ with the columns of $X$, 
which may lead to serious overestimation of the perturbation, especially in the case of random matrices.

\section{Simple proofs of the perturbation bounds of the extreme singular values}\label{main-results}

In this section, we give an alternative proof of Li-Li's inequality (\ref{li-li}) and obtain in passing a better lower bound. These bounds do not depend on the spacing ${\rm gap}_2$ unlike in 
\cite{IpsenNadler:SIAMMA09}.

\subsection{The maximum eigenvalue}

The following theorem provides sharp upper bounds for $\lb_{\max}(A)$, 
and lower bounds on $\lb_{\min}(A)$, depending on various information on the sub-matrix $M$ of $A$. As discussed 
above, this problem has close relationships with our problem of appending a column to a given rectangular matrix,
because $\lb_1(\wdt{A})=\lb_1(A)$. 

\begin{theo}[Li-Li's inequality and a lower bound]
\label{main}
Let $d$ be a positive integer and let $M\in \C^{d\times d}$ be an Hermitian matrix, whose eigenvalues are $\lambda_1\ge \cdots \ge \lambda_{d}$ with corresponding eigenvectors $(V_1,\cdots,V_d)$. Set $c\in \R$, $a\in \C^{d}$.
Let $A$ be given by (\ref{A}). Therefore:
\beq \label{lbmax}
\frac{2\la a,V_1\ra^2}{\eta_1+\sqrt{\eta_1^2+4\la a,V_1\ra^2}} \le \lb_{1}(A)-\max(c,\lb_1) \le  \frac{2\|a\|^2}{\eta_1+\sqrt{\eta_1^2+4\|a\|^2}},
\eeq
with
\bean
\eta_1 & = & |c-\lb_1|.
\eean
\end{theo}

\begin{rema} \ \rm
\begin{itemize}
\item Inequality (\ref{lbmax}) is sharp: the upper bound is reached when choosing $M=\Id$, $c=1$ and any $a$, so that $\lb_{\max}(A)=1+\|a\|$;
\item The lower bound in (\ref{lbmax}) is better than (\ref{li-li}) since we have:
\bean
\lb_{1}(A)\ge\max(c,\lb_1)+\frac{2\la a,V_1\ra^2}{\eta_1+\sqrt{\eta_1^2+4\la a,V_1\ra^2}} 
\ge \max(c,\lb_1)-\frac{2\|a\|^2}{\eta_1+\sqrt{\eta_1^2+4\|a\|^2}}.
\eean
Our lower bound is in particular consistent with Cauchy interlacing theorem, which states that $\lb_{1}(A)\ge \lb_1$. This lower bound can also be obtained by other methods as a Householder transformation, as pointed out by an anonymous referee; 
\item A great feature of Theorem 2 of Li and Li in \cite{Li-Li} is that it holds for all eigenvalues and for block perturbations. 
\end{itemize}
\end{rema}

\begin{proof}

Let $M=VDV^*$ denote the eigenvalue decomposition 
of $M$, i.e. $V=(V_1,\cdots,V_d)$ where the $V_i$'s are the orthonormal eigenvectors of $M$ and $D$ is a diagonal matrix whose diagonal entries are the real eigenvalues $\lambda_1\ge \cdots \ge \lambda_{d}$. 
We can write
\bean
A & = & 
\left(
\begin{array}{cc}
1 & 0  \\
0 & V
\end{array}
\right)
\left(
\begin{array}{cc}
c & a^*V  \\
V^* a & D
\end{array}
\right)
\left(
\begin{array}{cc}
1 & 0  \\
0 & V^*
\end{array}
\right),
\eean
and we set 
\bean
\label{blbl}
B & = & 
\left(
\begin{array}{cc}
c & b^*  \\
b & D
\end{array}
\right), \quad b = V^* a,
\eean
where we use the notation $b_j:=\la a,V_j\ra$.
Therefore, $A$ and $B$ have the same spectra and in particular,
\bea
\lb_{\max}(A) & = & \lb_{\max}(B).
\eea 

As in \cite{cc}, we compute the characteristic polynomial of the arrow matrix $B$:
\bean
P_B(\lb) & = & (c-\lb)\prod_{i=1}^d(\lb_i-\lb)- \sum_{i=1}^d\prod_{j\neq i}(\lb_j-\lb) b_j^2.
\eean
Let us define the function $f$ on $\R\setminus\{\lb_i,1\le i\le d\}$ as
\bean
f(\lb)  := P_B(\lb) \prod_{i=1}(\lb_i-\lb)^{-1} = c-\lb +\sum_{j=1}^{d} \frac{b_j^2}{\lb-\lb_j},
\eean
which is decreasing on $(\lb_1,+\infty)$ (even if $b=0$). 

We now assume that $b_1=\la a,V_1\ra\neq 0$. Thus $\lim_{\lb\to\lb_1}f(\lb)=+\infty$. From $\lim_{\lb\to+\infty}f(\lb)=-\infty$, we then deduce that the continuous function $f$ has a unique root on $(\lb_1,+\infty)$, that is 
\bean
\lb_{\max}(B) & > & \lb_1.
\eean

For all $\lb > \lb_1$, we have 
\bea \label{g}
f(\lb) & \le & c-\lb +  \frac{\|b\|_2^2}{\lb-\lb_1} := g(\lb).
\eea
For the same reasons as $f$, the function $g$ has a unique root $\lb^*$ on $(\lb_1,+\infty)$. 
Since $f$ is decreasing on $(\lb_1,+\infty)$ and $f(\lb_{\max}(B))=0=g(\lb^*)\ge f(\lb^*)$, we deduce:
\bean
\lb_{\max}(B) & \le & \lb^*.
\eean
We have
\bean
(\lb^*-c)(\lb^*-\lb_1) & = & \|b\|_2^2,
\eean
and thus $\lb^*$ is a root of the polynomial
\bean
Q(x) & = & (x-c)(x-\lb_1)-\|b\|^2\\
	& = & x^2 - (c+\lb_1) x + c\lb_1 - \|b\|_2^2. 
\eean
The discrimant of $Q$ reads:
\bean
\Delta & = & (c+\lb_1)^2 - 4(c\lb_1 - \|b\|_2^2) \\
 & =& (c-\lb_1)^2 + 4 \|b\|_2^2 >0.
\eean
Since $Q(\lb_1)<0$ and the dominant coefficient of $Q$ is positive, we deduce that $\lb^*$ is actually the greatest root of $Q$. Hence, noting that $\|b\|_2=\|a\|_2$,
\bea \label{lb*}
\lb^* & = &  \frac{c+\lb_1}{2}+\frac12 \sqrt{(c-\lb_1)^2+4\| a\|^2}.
\eea

Assume that $\la a,V_1\ra \neq 0$. In order to find a lower bound for $\lb_{\max}(B)$, we perform the same reasoning by writing
\bean
f(\lb) & \ge & c-\lb +  \frac{\la a,V_1\ra^2}{\lb-\lb_1},
\eean
and considering the polynomial $(x-c)(x-\lb_1)-\la a,V_1\ra^2$. 

Finally, we have:
\beq \label{}
\frac{c+\lb_1}{2}+\frac12 \sqrt{(c-\lb_1)^2+4\la a,V_1\ra^2} \le \lb_{1}(A) \le \frac{c+\lb_1}{2}+\frac12 \sqrt{(c-\lb_1)^2+4\| a\|^2}.
\eeq
Set $\eta_1= |c-\lb_1|$. Since
\bean
2\max(\lb_1,c) & = & c+\lb_1 + \eta_1,
\eean
we deduce 
\bean
\frac12\left(\sqrt{\eta_1^2+4\la a,V_1\ra^2}-\eta_1\right) \le \lb_{1}(A)-\max(\lb_1,c) \le \frac12\left(\sqrt{\eta_1^2+4\|a\|^2}-\eta_1\right).
\eean
Multiplying by the "conjugate quantity" yields the lower and the upper bounds in (\ref{lbmax}).\\

The case $\la a,V_1\ra= 0$ can be treated by standard continuity arguments: consider a continuous $\e\mapsto a(\e)$ such that 
for all $\e>0$, $\la a(\e),V_1\ra\neq0$ and $a(0)=a$. Ones then writes (\ref{lbmax}) for $\e>0$ and passes to the limit as $\e\to 0$.

\end{proof}

\begin{coro}[Weyl's inequality and Matthias' inequality]
In particular, the following simple perturbation bounds hold:
\bea
\lb_{1}(A) & \le &\max(c,\lb_{1}) + \|a\|_2 \label{ineg1max}\\
\lb_{1}(A) & \le &\max(c,\lb_{1}) + \frac{\|a\|_2^2}{|\lb_{1} - c|},   \label{ineg2max}
\eea
\end{coro}
\proof Inequality (\ref{ineg1max}) (resp. (\ref{ineg2max})) follows from (\ref{lbmax})
by using $\eta_1\ge 0$ (resp. $\|a\| \ge0$). \qed

\subsection{Perturbation of the smallest nonzero eigenvalue}

The same technics also allows to obtain lower bounds for the smallest nonzero eigenvalue, which are also direct consequences of Li-Li's inequality.

\begin{theo}
\label{smallestnonzero}
Let $d$ be a positive integer and let $M\in \C^{d\times d}$ be a positive semi-definite Hermitian matrix, whose eigenvalues are $\lambda_1\ge \cdots \ge \lambda_{d}$ with corresponding eigenvectors $(V_1,\cdots,V_d)$. Set $c\in \R$, $a\in \C^{d}$.
Let $A$ be given by (\ref{A}). Assume that $M$ has rank $r\le d$. Therefore:
\beq \label{lbmin}
\lb_{r+1}(A) \ge \min(c,\lb_r)- \frac{2\|a\|^2}{\eta_r+\sqrt{\eta_r^2+4\|a\|^2}},
\eeq
with
\bean
\eta_r & = & |c-\lb_r|.
\eean
\end{theo}

In particular, the following perturbation bounds of Weyl and Mathias hold:

\begin{coro}
\bea
\lb_{r+1}(A) & \ge &\min(c,\lb_{r}) - \|a\|_2 \label{ineg1min}\\
\lb_{r+1}(A) & \ge &\min(c,\lb_{r}) - \frac{\|a\|_2^2}{|c-\lb_{r}|}.   \label{ineg2min}
\eea
\end{coro}

\bigskip

\subsection{Bounds on the perturbation of the operator norm}\label{normop}

We provide here three bounds on the operator norm: the first and second inequalities are easy consequences of Theorem \ref{main}, the third one is based on a new trick.

\begin{cor}
Let $d$ be an integer, $a\in \C^{d}$, $c\in \R$ and let $M\in \C^{d\times d}$ be an Hermitian matrix. 
Let $A$ be given by (\ref{A}). Then the following inequalities hold:
\bea
\left\|A\right\| & \le & \max(c,\left\|M\right\|) + \|a\|_2  \label{ineg1'} \\
\left\|A\right\| & \le & \|M\| +  \frac{\|a\|_2^2}{\|M\|- c},\quad {\rm if}\ c\le \lb_{\max}(M) \label{ineg2'} \\
\left\|A\right\| & \le & \|M\| + \frac{|c|}{2} +\frac{\|a\|_2^2+c^2/8}{\|M\|}. \label{ineg3}
\eea
\end{cor}

\begin{rema}\rm 
Notice that (\ref{ineg2'}) is better than (\ref{ineg1'}) if
\bean
\|a\| & \le & \|M\|-c,
\eean
and that (\ref{ineg3}) is better than (\ref{ineg1'}) if
\bean
 \frac{c}{2}+ \frac{\|a\|_2^2+c^2/8}{\|M\|} & \le & \|a\|.
\eean
\end{rema}

\begin{proof}
We obtain (\ref{ineg1'}) by applying (\ref{ineg1max}) with $-A$ and by noticing that $\lb_{\max}(A) \le \|A\|$.

Now assume that $c\le \lb_{\max}(M)$. We bound $\Delta$ as:
\bean
\sqrt{\Delta} \le \sqrt{(\|M\|- c)^2 + 4 \|a\|_2^2},
\eean
and then
\bean
2\lb^* & \le &  2\|M\| +  \frac{2 \|a\|_2^2}{\|M\|- c},
\eean
which yields (\ref{ineg2'}).

To prove (\ref{ineg3}), we now consider, instead of $B$,
\bea
\label{blbl2}
B' & = & 
\left(
\begin{array}{ccc}
c & b^t & b^t \\
b & D & 0 \\
b &  0         & -D
\end{array}
\right).
\eea 
Since the operator norm increases by adding elements to a matrix , we obtain  
\bea
\label{blbl}
\left\|A \right\| & \le & \|B'\|
\eea 
The functions $f,g$ in (\ref{g}) are now replaced resp. by,
\bean 
\wdt f(\lb) & = & c-\lb +\sum_{j=1}^{d} b_j^2 \left(\frac1{\lb-\lb_j}+\frac1{\lb+\lb_j}\right) = c-\lb +\sum_{j=1}^{d} b_j^2 \ 
\frac{2\lb}{\lb^2-\lb_j^2} \\
\wdt g(\lb) & = & c-\lb +\|b\|_2^2 \ \frac{2\lb}{\lb^2-\|M\|^2}, \quad \lb>\|M\|.
\eean 
If $c\le 0$ then
\bean
\wdt f(\lb)\le\wdt g(\lb) \le \lb +\|b\|_2^2 \ \frac{2\lb}{\lb^2-\|M\|^2} :=h(\lb).
\eean
Let $x^*$ be a root of $h$. As previously, $\wdt f(\lb_{\max}(\wdt B))=0=h(x^*)\ge \wdt f(x^*)$, and then
\bean
\lb_{\max}(\wdt B) & \le & x^*.
\eean
But $x^*$ is less than the greatest root of the polynomial $x\mapsto x^2-\|M\|^2+2\|b^2\|$, that is:
\bean
x^* & \le & \sqrt{\|M\|^2+2 \|b\|_2^2}.
\eean
If $c>0$, we notice that 
\bean
(\lb^2-\|M\|^2)(c-\lb) + 2\lb \|b\|_2^2 & = & -\lb^3 +c \lb^2+(2\|b\|_2^2+\|M\|^2) \lb - c\|M\|^2 \\
	& \le & -\lb^3 +c \lb^2+(2\|b\|_2^2+\|M\|^2) \lb,
\eean
and we set
\bean
R(x) &= & x^2 -c x - (2\|b\|_2^2+\|M\|^2).
\eean
The greatest root $x^*$ of $R$ reads:
\bean
x^* & \le & \frac{c}{2}+\sqrt{\frac{c^2}{4}+\|M\|^2+2 \|b\|_2^2} \\
	 & \le &  \frac{c}{2}+ \|M\|\ \sqrt{1+\frac{2 \|b\|_2^2+c^2/4}{\|M\|^2}} \\
& \le & \frac{c}{2}+ \|M\| \ +\frac{\|b\|_2^2+c^2/8}{\|M\|}.
\eean  
Repeating the analysis with $-A$ yields (\ref{ineg3}) as desired.
\end{proof}

\bigskip

\section{Applications}
\label{app}

As already mentionned in the introduction, perturbations bounds on the extreme eigenvalues have many applications 
in science and engineering and some references were proposed. 
In this section, we focus two more applications where quadratic inequalities as the 
upper bound (\ref{lbmax}) can yield some improvements in the order of magnitude for the perturbed system.

\subsection{Restricted isometry constant and coherence in Compressed Sensing}

\subsubsection{General framework}
The purpose of Compressed Sensing (CS) is to study the various possible strategies 
for constructing efficient sensors allowing the recovery of very sparse signals in a high dimensional 
space (See e.g. the pioneering  work of Cand\`es, Romberg and Tao
\cite{CandesRombergTao:IEEEIT06}). The possiblity of building such types of sensors was first discovered through simulations in 
the study of Magnetic Resonnance Imaging, where sparsity in a certain dictionary was used in order to 
reconstruct the signal from much fewer measurements than was previously imagined. Since then, Compressed 
Sensing has found many applications as can be seen from the blog "Nuit Blanche" maintained by Igor Caron. 

The problem can be expressed mathematically as the one of solving the linear system 
\bean 
y & = & X \beta + \sigma \epsilon
\eean 
in the variable $\beta$, where $X\in \R^{n\times p}$, $\sigma\in \R_+$ and $\epsilon$ is a random noise. A major breakthrough 
occured in late 2005-early 2006 when \cite{CandesRombergTao:IEEEIT06}, \cite{CandesTao:IEEEIT06}, \cite{CandesRombergTao:CPAM06} 
and \cite{CandesTao:IEEEIT05} appeared. One of the main discoveries contained in these works 
is that the vector $\beta$ can be recovered exactly even when $p$ is much larger than $n$ 
and $n$ is as small as a constant times $s\log(p/s)$. The assumptions initially required that $\sigma=0$ and $\beta$ is $s$-sparse 
and the results were obtained for most $X$ drawn with i.i.d. components with standard gaussian or $\pm$1-Bernoulli distribution. It 
was then obtained in \cite{CandesRombergTao:CPAM06} and \cite{CandesPlan:AnnStat09} that the support 
of $\beta$ can be exactly recovered in the noisy case $\sigma > 0$ when $n$ is roughly of the same order. 
A basic property, which emerged from the analysis as a tool for proving the reconstructibility from few measurements, 
is the Restricted Isometry Property, which requires that all the submatrices $X_T$ have their singular values in the interval 
$[1-\rho,1+\rho]$ for some constant $\rho \in (0,1/2)$. Several authors \cite{Tropp:ACHA08}, \cite{Tropp:CRAS08}
and \cite{CandesPlan:AnnStat09} subsequently noticed that, assuming the columns of $X$ to be $\ell_2$-normalized, 
most submatrices $X_T$ obtained by selecting the columns indexed by $T$ with $|T|$ such that 
\bea 
|T| & \le & \frac{p}{\log p} \ \frac{C}{\|X\|^{2}} \label{ups}
\eea
for some constant $C$, have their singular values in the interval $[1-\rho,1+\rho]$ for some constant $\rho \in (0,1/2)$. 
Recall that the coherence $\mu(X)$ is defined by 
\bean
\mu(X) & = & \max_{j\neq j^\prime} |X_j^t X_{j^\prime}|.
\eean
This latter property can be interpreted in a probabilistic setting: let $T$ be a random subset of $\{1,\ldots,n\}$ 
drawn with uniform distribution over all subsets with cardinal bounded from above as in (\ref{ups}). 
Then, with high probability, $\|X_T^tX_T-I\|\le \rho$. 
\subsubsection{Perturbation of the singular values}
When an additional column is appended to the matrix $X$, one may wonder what is the impact of this operation on the 
localisation of the extreme singular values of all submatrices with $s$ columns which can be extracted from the resulting matrix. 
Notice that appending just one column to $X$ results in creating $p!/(s-1)!(p-s+1)!$ additional submatrices. 
Therefore, having a flexible bound on the perturbation of the extreme eigenvalues may be a valuable tool in practice. 
Another situation where perturbation has to be precisely controlled is  
when one wants to study the random variable $\|X^t_TX_T-I\|$ using the tools 
of modern concentration of measure theory \cite{BoucheronLugosiMassart:OUP13}. Indeed, after a 'Poissonization' 
trick has been employed as in Claim $(3.29)$ p.2173 in \cite{CandesPlan:AnnStat09}, one may study the problem 
on a product space for which the celebrated theorem of Talagrand or recent variants by Boucheron, Lugosi and Massart can be 
used. However, for such concentration theorems to be relevant, one also needs precise perturbation bounds on the 
extreme singular values.   

Let us consider the case where one uses a fixed design matrix $X$ and $T$ is obtained by selecting $s$ columns 
uniformly at random. Then, Lemma 3.6 
in \cite{CandesPlan:AnnStat09} implies that 
\bean 
\bP \left(\|X_T^tX_j\|_2^2 \ge s/p \|X\|^2+t\right) & \le & 2\exp\left(\frac{t^2}{2\mu^2(X)(s\|X\|^2/p+t/3)} \right)
\eean  
and thus, using (\ref{ups}), one easily obtains that 
\bea
\label{cp36}
\|X_T^tX_j\|_2^2 \le \frac1{4\log(p)}
\eea
with probability at least $1-2e^{-\frac{3}{64\mu^2(X)\log(p)}}$ if $C\le 1/8$. Assuming that the coherence is 
of the order of $1/\log(p)$, one obtains that (\ref{cp36}) holds with high probability. Thus, using inequality 
(\ref{ineg1max}), one obtains a perturbation of the order of $\log(p)^{-1/2}$ of the maximum eigenvalue of $X_T^tX_T$. On the other hand, if 
one is interested in the perturbation with norm already larger than $\sqrt{1+\rho}$, 
(\ref{ineg2max}) gives a perturbation of the norm of the order $\rho^{-1} \log(p)^{-1}$ which is significantly smaller 
and, as one might check in the assumptions of Theorem 5 in \cite{BoucheronLugosiMassart:AnnProba03}, is the 
right order of magnitude for obtaining the desired concentration of measure for this problem. 

\subsection{Perturbation of the algebraic connectivity of a graph by removing an edge}

Another application of spectral perturbation is in hypergraph theory. 
\subsubsection{The Laplacian of a graph}
The $G=(V,E)$ denote an oriented graph with vertex set $V$ and 
edge set $E$. In such a graph, each edge $e$ has a positive end and a negative end. We say that two 
vertices are adjacent if they are ends of the same edge. 
The indicence matrix $\mathcal I_G$ associated to $G$ is the matrix whose rows are indexed by the vertices and 
the columns are indexed by the oriented edges. The $(i,j)$-entry of $\mathcal I_G$ is 
\bean
\mathcal I_G(i,j) & = & 
\begin{cases}
+1 \textrm{ if vertex $i$ is the positive end of edge $j$ } \\
-1 \textrm{ if vertex $i$ is the negative end of edge $j$ } \\
0 \textrm{ otherwise.}
\end{cases}
\eean 
The adjacency matrix $\mathcal A_G$ is the matrix whose rows and columns are indexed by the vertices. The 
$(i,i^\prime)$-entry of $\mathcal A_G$ is
\bean
\mathcal A_G(i,i^\prime) & = & 
\begin{cases}
+1 \textrm{ if vertex $i$ and vertex $i^\prime$ are adjacent } \\
0 \textrm{ otherwise.}
\end{cases}
\eean 
The degree vector of $G$ is the vector $d_G$ where $d_G(i)$ is the number of edges of $G$ to which vertex $i$ is an end. 
The Laplacian matrix of $G$ is the matrix $\mathcal L_G$ defined by 
\bean 
\mathcal L_G & = & D(d_G)-\mathcal A_G,
\eean 
and the following well known identity holds 
\bea
\label{factolap}
\mathcal L_G & = & \mathcal I_G \mathcal I_G^t.  
\eea
If $G$ is not oriented, the degree vector and the adjacency matrix are defined in exactly the 
same way and any arbitrary orientation of the edges of $G$ will of course provide the same result. 
Notice that $\mathcal L_G$ is positive semi-definite and that $0$ is always an eigenvalue of $\mathcal L_G$. 
If the second smallest eigenvalue is nonzero, then the graph $G$ is connected. This second smallest eigenvalue 
is very important for the study of various graphs and is called the algebraic connectivity of $G$ or
Fiedler's value of $G$. We will denote the algebraic connectivity by $a(G)$. 
The eigenvalues of the Laplacian of a graph have been the subject of intense 
research for many years and is connected to various fields of pure and applied mathematics like expander families
\cite{HooryLinialWidgerson:BAMS06}, geometry of Banach spaces \cite{AlonMilman:JCT85}, Markov chains \cite{Bremaud:MarkovChains99}, clustering \cite{vonLuxburg:StatComp07}, to name just a few. 

\subsubsection{Edge deletion and the algebraic connectivity}
We now turn to the problem of controling the impact of deleting an edge on the algebraic connectivity of 
$\mathcal L$. The complement of a graph is the graph obtained by putting an edge between every non-adjacent 
couple of vertices and by deleting all edges already present in the graph before this operation. It is well
known \cite{Merris:LAA98} that 
\bea
\label{lowerfromupper}
a(G) & \ge & n-\lb_1(G^c). 
\eea 
Thus, controlling the effect of adding an edge to the complement of a graph allows to control the effect of 
deleting an edge of the graph on the algebraic connectivity. 

For $e=(u,v)$, with $u,v\in V(G)$, let $G^c+e$ denote the graph obtained from $G^c$ by appending the edge $e$. 
Let $i_e$ denote the column vector obtained by setting the component indexed by $u$ to -1 and 
the component indexed by $v$ to +1, and by setting all other components to zero.  
Since the Laplacian matrix $\mathcal L_{G^c}$ admits a factorization analogous to (\ref{factolap}), we obtain that  
$\mathcal L_{G^c}$ can be written in the form (\ref{A}) with $c=2$ and $a=\mathcal I_{G^c}^t i_e$. 

In many fields, it is very important to study the robustness of the graph topology to structural perturbations. For 
instance, the study of food webs has been of growing interest in the recent years \cite{Rossberg:Wiley2013}. As is well known, 
predation habits evolve with time as a consequence of landscape changes and competition. The 
world wide web is also an interesting application of graph theory and the formation and perturbation 
of communities is a topic of growing interest \cite{Newman:OUP10}. Communication 
systems are also often viewed as an interesting application of graph theory.  
In these examples, as in many other from ecology, social sciences, wireless communications, genetics, etc,
one is often interested in predicting the impact on topology of removing or adding an edge, a vertex or of various other
modifications of the structure, as measured by a relevant index such as the algebraic connectivity. 

\subsubsection{Controllability of complex networks}
In \cite{PorfiriDiBernardo:Automatica08}, the following model was proposed. One considers a set of 
$N$ $n$-dimensional oscillators governed by a system of nonlinear differential equations. 
Moreover, we assume that each oscillator is coupled with a restricted set of other oscillators. 
This coupling relationship can be efficiently described using a graph where the vertices 
are indexed by the oscillators and there is an edge between two oscillators if they are coupled. 
The overall dynamical system is given by the following set of differential equations
\bea \label{syst}
x_i^\prime(t) & = & f(x_i(t)) -\sigma B \sum_{j=1}^N l_{ij} x_j(t)+u_i(t), \ t\ge t_0,
\eea
$i=1,\ldots,N$, where $x_i(t) \in \R^n$ is the state of the $i^{th}$ oscillator, 
$\sigma$ is a positive real number, $B\in \R^{n\times n}$, 
$f: \R\mapsto \R$ describes the dynamics of each oscillator, $L=(l_{ij})_{i,j=1,\ldots,N}$ is the graph 
Laplacian of the underlying graph, and $u_i(t)$, $i=1,\ldots,N$ are the controls. 
For the system to be well defined, we have to specify some initial conditions $x_i(t_0)=x_{i0}$ for $i=1,\ldots,N$. 

Assume that we have a reference trajectory $s(t)$, $t\ge t_0$ satisfying the differential equation 
\bean 
s^\prime(t) & = & f(s(t)). 
\eean 
We want to control the system using a limited number of nodes. The selected nodes are called the "pinned nodes". 
For this purpose, we use a linear feedback law of the form 
\bean 
u_i(t) & = & p_i Ke_i(t), 
\eean 
where $e_i(t)=s(t)-x_i(t)$, $K$ is a feedback gain matrix, and where
\bean 
p_i & = & 
\begin{cases}
1 \: \textrm{ if node $i$ is pinned} \\
0 \: \textrm{ otherwise}.
\end{cases}
\eean 
Let $P$ denote the diagonal matrix with diagonal vector $p_1,\ldots,p_N$. 

The authors then give the definition of (global pinning-) controllability (based on Lyapunov stability criteria):

\begin{defi}
We say that the system (\ref{syst}) is controllable if the error dynamical system $e:=(e_i(t))_{1\le i\le N}$ is Lyapunov stable around the origin, i.e. there exists a positive definite function $V$ such that $\frac{d}{dt}V(e(t))< 0$ when $e(0)\neq 0$.
\end{defi}

The following result, \cite[Corollary 5]{PorfiriDiBernardo:Automatica08}, provide a sufficient condition for a system to be controllable:
\begin{prop}[\cite{PorfiriDiBernardo:Automatica08}]
\label{cor5}
Assume that $f$ is such that there exists a bounded matrix $F_{\xi,\wdt{\xi}}$, 
whose coefficients depend on $\xi$ and $\tilde{\xi}$, which satisfies 
\bea
\label{Fxi} 
F_{\xi,\wdt{\xi}}\left(\xi-\wdt{\xi}\right) & = & f(\xi)-f(\wdt{\xi}), \quad \xi,\wdt{\xi}\in\R^n.
\eea  
Let $Q\in \R^{n\times n}$ be a positive definite matrix such that 
\bean 
QK+K^tQ^t & = & \kappa \left(QB+B^tQ^t \right) \\
\left(QB+B^tQ^t \right) & \succeq 0 
\eean 
and 
\bea
\label{tart}
\frac12 \ \lb_{N} \left(\sigma L+\kappa P  \right)  \ \lb_{n}\left(QB+B^tQ^t \right) & > & \sup_{\xi,\wdt{\xi}}
\|F_{\xi,\wdt{\xi}}\|\ \|Q\|.
\eea 
Then the system is controllable.
\end{prop}
Many systems of interest satisfy the constraint specified by (\ref{Fxi}); see \cite{JiangTangChen:CSF03}.
This proposition is very useful for node selection via the matrix $P$. Indeed, assume that $Q$ is selected, 
then one may try to maximise $\lb_{N} \left(\sigma L+\kappa P  \right)$ as a function of $P$, 
under the constraint that no more than $r$ nodes can be pinned. This is a combinatorial problem that 
can be relaxed using semi-definite programming or various heuristics \cite{GhoshBoyd:IEEECDC06}. 

Using Theorem \ref{main}, we are in position for stating an easy controllability condition in the 
spirit of \cite[Corollary 7]{PorfiriDiBernardo:Automatica08}, based on the algebraic connectivity of the graph, 
the number of pinned nodes, the coupling strengh and the feedback gain. 
\begin{prop} \label{propkappa}
Let $Q\in \R^{n\times n}$ be a positive definite symetric matrix that satisfies  
\bean 
QK+K^tQ^t & = & \kappa \left(QB+B^tQ^t \right) \\
\left(QB+B^tQ^t \right) & \succeq 0,
\eean
and assume that
\bea
\label{Fcond}
\|F_{\xi,\tilde{\xi}}\| & < & \frac{\sigma \lb_{\min>0}(L) \ \lb_{\min}\left(QB+B^tQ^t\right) }{2\ \|Q\|}.
\eea 
If $\kappa$ satisfies 
\bean 
\kappa & \ge & \frac{\sum_{i=1}^r {\rm deg}_i}{\sigma \lb_{\min>0}(L) - \frac{2\ \|F_{\xi,\tilde{\xi}}\| \ \|Q\|}{\lb_{\min}\left(QB+B^tQ^t\right)}}
+\sigma \lb_{\min>0}(L),
\eean 
then the system is controllable. 
\end{prop}
\begin{proof}
We follow the same steps as for the proof of Corollary 7 in \cite{PorfiriDiBernardo:Automatica08}. 
We assume without loss of generality that the first $r$ nodes are the pinned nodes. 
We may write $P$ as 
\bean 
P & = & \sum_{i=1}^r e_i e_i^t,  
\eean 
where $e_i$ is the $i^{th}$ member of the canonical basis of $\R^N$, i.e. $e_i(j)=\delta_{i,j}$. We will 
try to compare $\lb_{N} \left(\sigma L+\kappa P \right)$ with $\lb_{N} \left(\sigma L\right)$ and 
use Proposition \ref{cor5} to obtain a sufficient condition for controllability based on $L$, i.e. 
the topology of the network. For this purpose, let us notice recall that $L$ can be written as 
\bean  
L & = & \mathcal I\cdot \mathcal I^t,
\eean 
where $\mathcal I$ is the incidence matrix of any directed graph obtained from the system's graph by 
assigning an arbitrary sign to the edges \cite{BrouwerHaemers:Springer12}. 
Of course $L$ will not depend on the chosen assignment. Using this factorization of $L$, 
we obtain that 
\bean 
\sigma L + \kappa \sum_{i=1}^r  e_ie_i^t & = & 
\left[\sqrt{\kappa}  \ e_r,\ldots,\sqrt{\kappa} \ e_{1},\sqrt{\sigma} \mathcal I\right] \ 
\left[\sqrt{\kappa} \ e_r,\ldots,\sqrt{\kappa}  \ e_{1},\sqrt{\sigma} \mathcal I\right]^t.
\eean  
Moreover, $\lb_{\min>0}\left(\sigma L + \kappa P\right)$ can be expressed easily as the 
smallest nonzero eigenvalue of the $r^{th}$ term of a sequence of matrices with shape (\ref{A}) 
for with we can use Theorem \ref{smallestnonzero} iteratively. Indeed, we have 
\bean 
\lb_{\min>0}\left( \sigma L + \kappa  e_1 \right)
& = & 
\lb_{\min>0}\left( \left[\sqrt{\kappa} \ e_{1},\sqrt{\sigma} \mathcal I\right]^t
\left[\sqrt{\kappa} \ e_{1},\sqrt{\sigma} \mathcal I\right]\right).
\eean 
Let us denote by $x$ the vector $\sqrt{\kappa} \ e_1$ and by $X$ the matrix 
$[\sqrt{\sigma} \mathcal I]$. Then, we have that 
\bean 
\left[\sqrt{\kappa} \ e_{1},\sqrt{\sigma} \mathcal I\right]^t
\left[\sqrt{\kappa} \ e_{1},\sqrt{\sigma} \mathcal I\right] 
& = & 
\left[
\begin{array}{cc}
x^t x & x^t X \\
X^tx  & X^t X
\end{array}
\right].
\eean 
Therefore, Theorem \ref{smallestnonzero} gives 
\bean 
\lb_{\min>0} \left(\sigma L+\kappa e_1e_1^t \right) & \ge & \sigma \lb_{\min>0}(L) 
-\frac{{\rm deg}_1}{(\kappa-\sigma \lb_{\min>0}(L))},
\eean 
where ${\rm deg}_1$ is the degree of node number 1. 

Let us now consider 
$\lb_{\min>0} \left(\sigma L + \kappa \ e_1 +\delta_2 e_2 \right)$. We have that 
\bean 
\lb_{\min>0} \left(\sigma L + \kappa \ e_1 +\delta_2 e_2 \right)
& = & \lb_{\min>0} \left(
\left[\sqrt{\kappa} \ e_2,\sqrt{\kappa} \ e_{1},\sqrt{\sigma} \mathcal I\right]^t
\left[\sqrt{\kappa} \ e_2,\sqrt{\kappa} \ e_{1},\sqrt{\sigma} \mathcal I\right]\right).
\eean 
Let us denote by $x$ the vector $\sqrt{\kappa} \ e_2$ and by $X$ the matrix 
$[\sqrt{\kappa} \ e_1,\sqrt{\sigma} \mathcal I]$. Then, we have that 
\bean 
\left[\sqrt{\kappa} \ e_{2},\sqrt{\kappa} \ e_{1},\sqrt{\sigma} \mathcal I\right]^t
\left[\sqrt{\kappa} \ e_{2},\sqrt{\kappa} \ e_{1},\sqrt{\sigma} \mathcal I\right] 
& = & 
\left[
\begin{array}{cc}
x^t x & x^t X \\
X^tx  & X^t X
\end{array}
\right]
\eean
and using Theorem \ref{smallestnonzero} again, we obtain  
\bean 
\lb_{\min>0} \left(\sigma L+\kappa e_1e_1^t+\kappa e_2e_2^t \right) & \ge & \lb_{\min>0}(\sigma L+\kappa e_1e_1^t) 
-\frac{{\rm deg}_2}{(\kappa-\lb_{\min>0}(\sigma L+\kappa e_1e_1^t))}. 
\eean  
Since $\lb_{\min>0}(\sigma L+\kappa e_1e_1^t) \le \lb_{\min>0}(\sigma L)$, we thus obtain 
\bean 
\lb_{\min>0} \left(\sigma L+\kappa e_1e_1^t+\kappa e_2e_2^t \right) & \ge & \lb_{\min>0}(\sigma L+\kappa e_1e_1^t) 
-\frac{{\rm deg}_2}{(\kappa-\sigma \lb_{\min>0}(L))}. 
\eean  

We can repeat the same argument $r$ times and obtain 
\bea
\label{bnd>0} 
\lb_{\min>0} \left(\sigma L+\kappa P \right) & \ge & \sigma \lb_{\min>0}(L) -
\frac{\sum_{i=1}^r {\rm deg}_i}{\kappa-\sigma \lb_{\min>0}(L)}. 
\eea 

Finally, by Proposition \ref{cor5}, we know that the following constraint is sufficient for 
preserving controllability 
\bea
\label{condkappa}
\lb_{\min>0}\left(\sigma L+\kappa \sum_{i=1}^r e_ie_i^t\right) & \ge  & \frac{2\ \|F_{\xi,\tilde{\xi}}\| \ \|Q\|}{\lb_{\min}\left(QB+B^tQ^t\right)}.
\eea
By (\ref{bnd>0}), it is sufficient to garantee the controllability of our system to impose 
\bean 
\sigma \lb_{\min>0}(L) -
\frac{\sum_{i=1}^r {\rm deg}_i}{\kappa-\sigma \lb_{\min>0}(L)} & \ge & \frac{2\ \|F_{\xi,\tilde{\xi}}\| \ \|Q\|}{\lb_{\min}\left(QB+B^tQ^t\right)}.
\eean 
Then, combining (\ref{condkappa})  with (\ref{Fcond}) implies that 
\bean 
\kappa & \ge & \frac{\sum_{i=1}^r {\rm deg}_i}{\sigma \lb_{\min>0}(L) - \frac{2\ \|F_{\xi,\tilde{\xi}}\| \ \|Q\|}{\lb_{\min}\left(QB+B^tQ^t\right)}}
+\sigma \lb_{\min>0}(L)
\eean 
is a sufficient condition for controllability.
\end{proof}

{\bf Acknowledgement --} We thank Maurizio Porfiri for pointing out a missing assumption in 
Proposition \ref{propkappa}.



\bibliographystyle{amsplain}
\bibliography{database}

\end{document}